\documentclass{amsart}

\usepackage{amsmath,amsthm, amssymb, amsfonts, epsfig,amscd}
\usepackage[all]{xy}

\usepackage[mathscr]{eucal}
\usepackage{mathrsfs}
\usepackage{latexsym}
\usepackage{graphicx}
\theoremstyle{plain}
\newtheorem{thm}[subsection]{Theorem}
\newtheorem{lem}[subsection]{Lemma}
\newtheorem{prop}[subsection]{Proposition}

\newtheorem{lemma}[subsection]{Lemma}
\newtheorem{theorem}[subsection]{Theorem}

\theoremstyle{definition}
\newtheorem{rem}[subsection]{Remark}
\newtheorem{defn}[subsection]{Definition}
\newtheorem{definition}[subsection]{Definition}
\newtheorem{remark}[subsection]{Remark}

\def\vs{\vskip}
\def\ni{\noindent}

\begin{document}
\title[Wahl's conjecture for a minuscule $G/P$]%
{Wahl's conjecture for a minuscule $G/P$}
\author[J. Brown]{J. Brown}
\address{Northeastern University, Boston, USA}
   \email{brown.justin1@neu.edu}
\author[V. Lakshmibai]{V. Lakshmibai${}^{\dag}$}
\address{Northeastern University, Boston, USA}
   \email{lakshmibai@neu.edu}

\thanks{${}^{\dag}$ V.~Lakshmibai was partially supported by NSF grant DMS-0652386 and
Northeastern University RSDF 07--08.}
 \maketitle

%\title{Wahl's Conjecture}
%\maketitle
\begin{abstract}%
%\vskip-1cm
We show that Wahl's conjecture holds in all characteristics for a
minuscule $G/P$.
%\vskip.2cm
\end{abstract}
%\maketitle
%\input{introduction}
Let $X$ be a non-singular projective variety over $\mathbb{C}$.
For ample line bundles $L$ and $M$ over $X$, consider the natural
restriction map (called the {\em Gaussian\/})
\begin{equation*}\label{e.gaussian}
 H^0(X\times X, \mathcal{I}_\Delta\otimes p_1^*L\otimes p_2^*M)\to
H^0(X, \Omega^1_{X}\otimes L\otimes M)
\end{equation*}
where $\mathcal{I}_\Delta$ denotes the ideal sheaf of the diagonal
$\Delta$ in $X\times X$, $p_1$ and $p_2$ the two projections of
$X\times X$ on $X$, and $\Omega^1_{X}$ the sheaf of differential
$1$-forms of~$X$; note that this map is induced by the natural
projection
$\mathcal{I}_\Delta\rightarrow\rightarrow\mathcal{I}_\Delta/\mathcal{I}^2_\Delta$
by identifying the $\mathcal{O}_\Delta$-module
$\mathcal{I}_\Delta/\mathcal{I}^2_\Delta$ with $\Omega^1_{X}$.
Wahl conjectured in~\cite{wahl} that this map is surjective when
 $X=G/P$ for $G$ a complex semisimple algebraic group and  $P$ a
parabolic subgroup of $G$. Wahl's conjecture was proved by Kumar~
in \cite{kumar} using representation theoretic techniques.
In~\cite{lmp}, the authors considered Wahl's conjecture in
positive characteristics, and observed that Wahl's conjecture will
follow if there exists a Frobenius splitting of $X\times X$ which
compatibly splits the diagonal and which has the maximum possible
order of vanishing along the diagonal; this stronger statement was
formulated as a conjecture in \cite{lmp} (see \S\ref{s.blowups}
for a statement of this conjecture) which we shall refer to as the
\emph{LMP-conjecture} in the sequel. Subsequently, in \cite{mp},
Mehta-Parameswaran proved the LMP-conjecture for the Grassmannian.
Recently, Lakshmibai-Raghavan-Sankaran (cf.\cite{lrs}) extended
the result of \cite{mp} to symplectic and orthogonal
Grassmannians. In this paper, we show that the LMP conjecture (and
hence Wahl's conjecture)
 holds in all characteristics for a minuscule $G/P$ (of
course, if $G$ is the special orthogonal group $SO(m)$, then one
should not allow characteristic $2$). The main philosophy of the
proof is the same as in \cite{mp,lrs}; it consists in reducing the LMP
conjecture for a $G/P,P$ a parabolic subgroup to the problem of
finding a section $\varphi\in H^0(G/B,K^{-1}_{G/B})$ ($K_{G/B}$
being the canonical bundle on $G/B$) which has maximum possible
order of vanishing along $P/B$. This problem is further reduced to
computing the order of vanishing (along $P/B$) of the highest
weight vector $f_d$ in $H^0(G/B,L(\omega_d))$, for every
fundamental weight $\omega_d$ of $G$. For details, see \S
\ref{steps}.

%Kumar~\cite{kumar} showed,  using representation theoretic techniques,
%that the conjecture holds in characteristic zero.
 It should be remarked that though the spirit of this paper
is the same as that of \cite{mp,lrs}, the methods used (for computing
the order of vanishing of sections) in this paper differ from
those of \cite{mp,lrs}. Of course, the methods used in this paper
may also be used for proving the results of \cite{mp,lrs}. Thus
our methods provide an alternate proof of the results of
\cite{mp,lrs}; we have given the details in \S \ref{last}.

As a by-product of our methods, we obtain a nice combinatorial
realization of the order of vanishing (along $P/B$) of $f_d$ as
being the length of the shortest path through extremal weights in
the weight lattice connecting the highest weight (namely,
$i(\omega_d), i$ being the Weyl involution) in
$H^0(G/B,L(\omega_d))$ and the extremal weight $-\tau(\omega_d),
\tau$ being the element of largest length in $W_P$, the Weyl group
of $P$ (see Remark \ref{path}, Remark \ref{path2}).

This paper is organized as follows: In \S \ref{one}, we fix
notation. In \S \ref{s.splittings}, we recall some basic
definitions and results about Frobenius splittings, and also the
canonical section $\sigma\in H^0(G/B, K^{1-p})$. In
\S\ref{s.blowups}, we recall the results of~\cite{lmp} about
splittings for blow-ups and also the LMP conjecture. In \S
\ref{steps}, we describe the steps leading to the reduction of the
proof of the LMP conjecture to computing $ord_{P/B}\sigma$ (the
order of vanishing of $\sigma$ along $P/B$). In \S \ref{follow}, a
further reduction is carried out. In \S \ref{dn}, \S \ref{e6},\S
\ref{e7}, the details are carried out for \textbf{D}$_n$,
\textbf{E}$_6$, \textbf{E}$_7$ respectively. In \S \ref{last}, we
give the details for the remaining minuscule $G/P$'s.

\vs.2cm\ni\textbf{Acknowledgment:} Part of the work in this paper
was carried out when the second author was visiting University of
K\"oln during May-June, 2008; the second author expresses her
thanks to University of K\"oln for the hospitality extended to her
during her visit.
%\vfill\eject
%In~\S\S\ref{s.splittings},~\ref{s.canonical},  we recall some basic
%definitions and specify what we mean by the `canonical splittings' of $G/B$,
%$G/P$, and $G/B\times G/B$ (notation is fixed in~\S\ref{s.notation}).
%Our main result (Theorem~\ref{t.main}) says that the

\section{Notation}\label{one} Let $k$ be the base field which we assume to be
algebraically closed of positive characteristic; note that if
Wahl's conjecture holds in infinitely many positive
characteristics, then it holds in characteristic zero also, for
the Gaussian is defined over the integers. Let $G$ be a simple
algebraic group over $k$ (if $G$ is the special orthogonal group,
then characteristic of $k$ will be assumed to be different from
$2$). Let $T$ be a maximal torus in $G$, and $R$ the root system
of $G$ relative to $T$. We fix a Borel subgroup $B,B\supset T$;
let $S$ be the set of simple roots in $R$ relative to $B$, and let
$R^+$ be the set of positive roots in $R$. We shall follow
\cite{bou} for indexing the simple roots. Let $W$ be the Weyl
group of $G$; then the $T$-fixed points in $G/B$ (for the action
given by left multiplication) are precisely the cosets
$e_w:=wB,w\in W$. For $w\in W$, we shall denote the associated
Schubert variety (the closure of the $B$-orbit through $e_w$) by
$X(w)$ .

\section{Frobenius Splittings}\label{s.splittings} Let $X$ be a
scheme over~$k$, separated and of finite type. Denote by $F$ the
{\em absolute Frobenius\/} map on $X$: this is the identity map on
the underlying topological space~$X$ and is the $p$-th power map
on the structure sheaf $\mathcal{O}_X$. We say that $X$ is {\em
Frobenius split\/}, if the $p$-th power map
$F^\#:\mathcal{O}_X\rightarrow F_*\mathcal{O}_X$ splits as a map
of $\mathcal{O}_X$-modules (see~\cite[\S1,~Definition~2]{mr},
\cite[Definition~1.1.3]{bk}). A splitting $\sigma:
F_*\mathcal{O}_X\to\mathcal{O}_X$ {\em compatibly splits\/} a
closed subscheme~$Y$ of $X$ if $\sigma(F_*\mathcal{I}_Y)\subseteq
\mathcal{I}_Y$ where $\mathcal{I}_Y$ is the ideal sheaf of~$Y$
(see \cite[\S1,~Definition~3]{mr}, \cite[Definition~1.1.3]{bk}).

% \begin{proposition}
%  Let $f:X\to Y$ be a morphism of schemes such that
% $f^\#:f_*\ox\leftarrow\oy$ is an isomorphism.    Then the direct image
% $f_*$ of a splitting of $X$ gives a splitting of $Y$.    If $f:X\to Y$
% is in addition birational,  then two splittings of $X$ that induce the
% same splitting on $Y$ must be equal.
% \end{proposition}
Now let $X$ be a non-singular projective variety, and $K$ its
canonical bundle. Using Serre duality (and the observation that
$F^*L\cong L^p$ for an invertible sheaf~$L$ on $X$),  we get a
($k$-semilinear) isomorphism of
$H^0(X,\mathcal{H}om(F_*\mathcal{O}_X,\mathcal{O}_X))$

\ni $(=Hom_{\mathcal{O}_X} (F_*\mathcal{O}_X, \mathcal{O}_X))$
with $H^0(X,K^{1-p})$ (see~\cite[Page~32]{mr}, \cite[Lemma~1.2.6
and \S1.3]{bk}). Thus to find splittings of $X$, we are led to
look at a $\sigma$ in $H^0(X,K^{1-p})$ such that the associated
homomorphism $F_*\mathcal{O}_X\to \mathcal{O}_X$ is a splitting of
$F^\#$; in the sequel, following \cite{mr}, we shall refer to this
situation by saying \emph{the element $\sigma\in H^0(X,K^{1-p})$
splits $X$}.

\begin{remark}\label{max}
By local computations, it can be seen easily that if a $\sigma\in
H^0(X,K^{1-p})$ vanishes to order $>d(p-1)$ along a subvariety $Y$
of codimension $d$ for some $1\le d\le dim\,X -1$, then $\sigma$
is not a splitting of $X$. Hence we say that \emph{a subvariety
$Y$ is compatibly split by $\sigma$ with maximum multiplicity} if
$\sigma$ is a splitting of $X$ which compatibly splits $Y$ and
which vanishes to order $d(p-1)$ generically along $Y$.
\end{remark}

We will often use the following Lemma:
\begin{lemma}\label{often}
Let $f:X\rightarrow Y$ be a morphism of schemes such that $f^\#:
\mathcal{O}_Y\rightarrow f_*\mathcal{O}_X$ is an isomorphism.
\begin{enumerate}
\item If $X$ is Frobenius split, then so is $Y$. \item If $Z$ is
compatibly split, then so is the scheme-theoretic image of $Z$ in
$Y$.
\end{enumerate}
\end{lemma}
For a proof, see \cite{bk},Lemma 1.1.8 or \cite{mr}, Proposition
4.
\subsection{The section $\sigma\in H^0(X,K^{-1}_X)$}\label{sect}
As above, let $X=G/B$. In \cite{mr}, a section $s'\in
H^0(X,K^{1-p}_X)$ giving a splitting for $X$ is obtained by
inducing it from a section $s\in H^0(Z,K^{1-p}_Z)$ which gives a
splitting for $Z$, the Bott-Samelson variety. It turns out  that
up to a non-zero scalar multiple, $s'$ equals $\sigma^{p-1}$ where
$\sigma\in H^0(X,K^{-1}_X)$. In fact, one has an explicit
description of $\sigma$: We have, $K^{-1}_X=L(2\rho)$ where $\rho$
denotes half the sum of positive roots (here, for an integral
weight $\lambda,L(\lambda)$ denotes the associated line bundle on
$X$). Let $f^+,f^-$ denote respectively a highest, lowest weight
vector in $H^0(X,L(\rho))$ (note that $f^+,f^-$ are unique up to
scalars). Then $\sigma$ is the image of $f^+\otimes f^-$ under the
 map \[H^0(X,L(\rho))\otimes H^0(X,L(\rho))\rightarrow
 H^0(X,L(\rho))\] given by multiplication of sections.

 See \cite[\S 2.3]{bk} for details.
 \section{Splittings and Blow-ups}\label{s.blowups}
Let $Z$ be a non-singular projective variety and $\sigma$ a
section of $K^{1-p}$ (where $K$ is the canonical bundle) that
splits~$Z$. Let~$Y$ be a closed non-singular subvariety of~$Z$ of
codimension~$c$.   Let $ord_Y\sigma$ denote the order of vanishing
of $\sigma$ along $Y$. Let $\pi:\tilde{Z}\rightarrow Z$ denote the
blow up of $Z$ along~$Y$ and~$E$ the exceptional divisor (the
fiber over~$Y$) in~$\tilde{Z}$.
%As is easily seen
%(\cite[Proposition~4]{mr}, \cite[Proposition~1.1.8]{bk}),

A splitting $\tilde{\tau}$ of $\tilde{Z}$ induces a
splitting~$\tau$ on~$Z$, in view of Lemma \ref{often} (since,
$\pi_*\mathcal{O}_Z\rightarrow \mathcal{O}_{\tilde{Z}}$ is an
isomorphism). We say that~$\sigma$ {\em lifts\/} to a splitting of
${\tilde{Z}}$ if it is induced thus from a
splitting~$\tilde{\sigma}$ of~${\tilde{Z}}$ (note that the lift of
$\sigma$ to ${\tilde{Z}}$ is unique if it exists, since
${\tilde{Z}}\to Z$ is birational and two global sections of the
locally free sheaf
$\mathcal{H}om_{\mathcal{O}_{\tilde{Z}}}(F_*\mathcal{O}_{\tilde{Z}},
\mathcal{O}_{\tilde{Z}})$ that agree on an open set must be
equal).

\begin{prop}
With notation as above, we have
\begin{enumerate}
\item $ord_Y\sigma\leq c(p-1)$. \item If $ord_Y\sigma=c(p-1)$ then
$Y$ is compatibly split. \item $ord_Y\sigma \ge(c-1)(p-1)$ if and
only if $\sigma$ lifts to a splitting~$\tilde{\sigma}$
of~$\tilde{Z}$; moreover, $ord_Y\sigma = c(p-1)$ if and only if
the splitting $\tilde{\sigma}$ is compatible with $E$.
\end{enumerate}
\end{prop}
\begin{proof}
Assertion(1) follows in view of Remark \ref{max} (since $\sigma$
is a splitting). Assertion (2) follows from the local description
as in~\cite[Proposition~5]{mr}. For a proof of assertion (3),
see \cite{lmp}, Proposition~2.1.
\end{proof}

Now let $Z=G/P\times G/P$, and $Y$ the diagonal copy of $G/P$
in~$Z$. We have:
\begin{theorem}[cf.%Lakshmibai, Mehta, and Parameswaran~
\cite{lmp}]\label{t.lmp} Assume that the characteristic~$p$ is odd. If $E$ is
compatibly split in $\tilde{Z}$, or, equivalently, if there is a
splitting of $Z$ compatibly splitting~$Y$ with maximal
multiplicity, then the Gaussian map is surjective for~$X=G/P$.
\end{theorem}

Let us recall (cf.\cite{lmp}) the following conjecture:

\noindent\textbf{LMP-Conjecture}
    For any $G/P$, there exists a splitting of $Z$ that
compatibly splits the diagonal copy of~$G/P$ with maximal
multiplicity.
\section{Steps leading to a proof of LMP-conjecture for a minuscule
$G/P$}\label{steps} Our proof of the LMP-conjecture for a minuscule
$G/P$ is in the same spirit as in \cite{mp}. We describe below a
sketch of the proof.

\vs.2cm\ni I. \textbf{The splitting $\lambda$ of $G\times^B G/B:$}
For a Schubert variety $X$ in $G/B$, using the $B$-action on $X$,
we may form the twisted fiber space
$$G\times^B X:=G\times
X/(gb,b^{-1}x)\sim(g,x),\ g\in G,b\in B, x\in X$$ For $X=G/B$, we
have a natural isomorphism
\[f:G\times^B G/B\cong G/B\times G/B,\,(g,xB)\mapsto(gB,gxB)\]
 We have
(cf.\cite{mr2}) that there exists a splitting for $G\times^B G/B
(\cong G/B\times G/B)$ compatibly splitting the $G$-Schubert
varieties $G\times^B X$. In fact, by \cite{bk}, Theorem 2.3.8, we
have that this splitting is induced by $\sigma^{p-1}$ (where
$\sigma$ is as in \S\ref{sect}; as in that subsection, one
identifies $\sigma$ with $f^+\otimes f^-$). We shall denote this
splitting of $G\times^B G/B$ by $\lambda$.

\vs.2cm\ni II. \textbf{Order of vanishing of $\lambda$ along
$G\times^B P/B$:} Let $P$ be a (standard) parabolic subgroup. From
the description of $\lambda$, it is clear that the order of
vanishing of $\lambda$ along $G\times^B P/B$ equals $(p-1)ord_{P/B
}\sigma$, where $ord_{P/B} \sigma$ denotes the order of vanishing
of $\sigma$ along $P/B$. For simplicity of notation, let us denote
this order by $q$.

\vs.2cm\ni III. \textbf{Reduction to computing the order of
vanishing of $\sigma$ along $P/B$:} Consider the natural
surjection $\pi:G/B\times G/B\rightarrow G/P\times G/P,
(g_1B,g_2B)\mapsto (g_1P,g_2P)$. Then under the identification
$f:G\times^B G/B\cong G/B\times G/B$, we have that $\pi$ induces a
surjection
\[G\times^B P/B\rightarrow \Delta_{G/P}\] where $\Delta_{G/P}$
denotes the diagonal in $G/P\times G/P$. We now recall the
following Lemma from \cite{lmp}
\begin{lem}
Let $f:X\rightarrow Y$ be a morphism of schemes such that $f^\#:
\mathcal{O}_Y\rightarrow f_*\mathcal{O}_X$ is an isomorphism. Let
$X_1$ be a smooth subvariety of $X$ such that $f$ is smooth
(submersive) along $X_1$. If $X_1$ is compatibly split in $X$ with
maximum multiplicity, then the induced splitting of $Y$ has
maximum multiplicity along $f(X_1)$.
\end{lem}

\vs.2cm\ni\textbf{Main Reduction:} Hence the LMP-conjecture will hold
for a $G/P$ if we could show that $q$ equals $(p-1)dim\,G/P$,
equivalently that $ord_{P/B} \sigma$ equals
$dim\,G/P$

\ni $(=codim_{G/B}P/B)$.

\begin{definition} A fundamental weight $\omega$ is called
\emph{minuscule} if
$\left<\omega,\beta\right>(={\frac{2(\omega,\beta)}{(\beta,\beta)}})\leq
1$ for all $\beta\in R^+$; the maximal parabolic subgroup
associated to $\omega$ is called a \emph{minuscule parabolic
subgroup}.
\end{definition}

In the following sections, we prove that $ord_{P/B} \sigma$ equals
$dim\,G/P$ for a minuscule $G/P$. This is in fact the line of
proof for the Grassmannian in \cite{mp}, and for the symplectic
and orthogonal Grassmannians in \cite{lrs}; as already mentioned,
our methods (for computing the order of vanishing of sections)
differ from those of \cite{mp,lrs}. In the following section, we
describe the main steps involved in our approach; but first we
include the list of all of the minuscule fundamental weights,
following the indexing of the simple roots as in \cite{bou}:
\[ \begin{array}{lll}
        \mbox{Type } \mathbf{A_n} & : & \mbox{Every fundamental weight is minuscule}\\
        \mbox{Type } \mathbf{B_n} & : & \omega_n\\
        \mbox{Type } \mathbf{C_n} & : & \omega_1\\
        \mbox{Type } \mathbf{D_n} & : & \omega_1,\omega_{n-1},\omega_n\\
        \mbox{Type } \mathbf{E_6} & : & \omega_1, \omega_6\\
        \mbox{Type } \mathbf{E_7} & : & \omega_7.
\end{array}\]
There are no minuscule weights in types $\mathbf{E_8},
\mathbf{F_4}, \mbox{ or } \mathbf{G_2}.$

\section{Steps leading to the determination of $ord_{P/B}\sigma $}\label{follow} We first
describe explicit realizations for $f^+,f^-$, and then describe
the main steps involved in computing the the order of vanishing of
$\sigma$ along $P/B$.

\vs.2cm\ni\textbf{Explicit realizations for $f^+,f^-$:}  We shall
denote a maximal parabolic subgroup corresponding to omitting a
simple root $\alpha_i$ by $P_i$; also, we follow the indexing of
simple roots, fundamental weights etc., as in \cite{bou}. Let
$\omega_1,\cdots,\omega_l$ be the fundamental weights ($l$ being
the rank of $G$). For $1\le d\le l$, let $V(\omega_d)$ be the Weyl
module with highest weight $\omega_d$. One knows (see \cite{jan}
for instance) that the multiplicity of $\omega_d$ (in
$V(\omega_d)$) is $1$. Then $w(\omega_d),w\in W$ give all the
\emph{extremal weights} in $V(\omega_d)$, and these weights again
have multiplicities equal to $1$; of course, it suffices to run
$w$ over a set of representatives of the elements of
$W/W_{P_{d}}$. Given $w\in W$, let us fix representatives
$w^{(d)},1\le d\le l$ for $wW_{P_{d}},1\le d\le l$. Let us fix a
highest weight vector in $V(\omega_d)$ and denote it by
$q_{e^{(d)}}$; denote $w^{(d)}.q_{e^{(d)}}$ by $q_{w^{(d)}}$. Note
in particular that $q_{w_0^{(d)}}$ is a lowest weight vector,
$w_0$ being the element of largest length in $W$. Recall the
following well-known fact (see \cite{jan}) \[H^0(G/P_d,
L(\omega_d))\cong V(\omega_d)^*\] where $V(\omega_d)^*$ is the
linear dual of $V(\omega_d)$. In particular,  $H^0(G/P_d,
L(\omega_d))$ may be identified with the Weyl module
$V(i(\omega_d)),i$ being the Weyl involution (equal to $-w_0$, as
an element of Aut$\,R$), and thus the extremal weights in
$H^0(G/P_d, L(\omega_d))$ are given by $-w^{(d)}(\omega_d)$. We
may choose extremal weight vectors $p_{w^{(d)}}$ in

\ni $H^0(G/P_d, L(\omega_d))$, of weight $-w^{(d)}(\omega_d)$, in
such a way that under the canonical $G$-invariant bilinear form
$(,)$ on $H^0(G/P_d, L(\omega_d))\times V(\omega_d)$, we have,
\[(p_{\theta^{(d)}},q_{\tau^{(d)}})=
\delta_{\theta^{(d)},\tau^{(d)}}\] As a consequence, we have, for
$\tau\in W$, \[p_{\theta^{(d)}}\,|_{X(\tau)}\ne 0\
\Leftrightarrow\  \theta^{(d)}\in X(\tau^{(d)})\leqno{(*)}\] Now
$\rho$ being $\omega_1+\cdots +\omega_l$, we may take $f^+$ (resp.
$f^-$) to be the image of $f^+_1\otimes\cdots\otimes f^+_l$ (resp.
$f^-_1\otimes\cdots\otimes f^-_l$) under the canonical map
\[H^0(G/B, L(\omega_1))\otimes\cdots\otimes H^0(G/B, L(\omega_l))
\rightarrow H^0(G/B, L(\rho))\] given by multiplication of
sections. Hence we may choose \[f^+={\underset{1\le d\le
l}{\prod}}\,p_{w_0^{(d)}},\ f^-={\underset{1\le d\le
l}{\prod}}\,p_{e^{(d)}} .\] Thus, $\sigma$ may be taken to be
\[\sigma=({\underset{1\le d\le
l}{\prod}}\,p_{w_0^{(d)}})({\underset{1\le d\le
l}{\prod}}\,p_{e^{(d)}}) .\] Now $eB$ belongs to every Schubert
variety, and hence in view of $(*)$, $p_{e^{(d)}}\,|_X\ne 0,1\le
d\le l$, for any Schubert variety $X$. In particular,
\[p_{e^{(d)}}\,|_{P/B}\ne 0,1\le d\le l .\] Hence we obtain
\[ord_{P/B}\,\sigma
=ord_{P/B}({\underset{1\le d\le
l}{\prod}}\,p_{w_0^{(d)}})={\underset{1\le d\le
l}\sum}\,ord_{P/B}\,p_{w_0^{(d)}} .\]

Thus we are reduced to computing $ord_{P/B}\,p_{w_0^{(d)}}$.

\vs.2cm\ni\textbf{Computation of $ord_{P/B}\,p_{w_0^{(d)}}$:} Let
us denote the element of largest length in $W_P$ by $\tau_P$ or
just $\tau$ ($P$ having been fixed). Since the $B$-orbit through
$\tau$ (we are denoting $e_\tau$ by just $\tau$) is dense open in
$P/B$, we have
\[ord_{P/B}\,p_{w_0^{(d)}}=ord_{\tau}\,p_{w_0^{(d)}}\] where the
right hand side denotes the order of vanishing of $p_{w_0^{(d)}}$
at the point $e_\tau$. Hence \[ord_{P/B}\,\sigma ={\underset{1\le
d\le l}\sum}\,ord_{\tau}\,p_{w_0^{(d)}} .\] Thus, our problem is
reduced to computing $ord_{\tau}\,p_{w_0^{(d)}}$; to compute this,
we may as well work in $G/P_{d}$. We shall continue to denote the
point $\tau\,P_{d}$ (in $G/P_{d}$) by just $\tau$.

The affine space $\tau B^-\tau^{-1}\cdot \tau\,P_{d}$ ($B^-$ being
the Borel subgroup opposite to $B$) is open in $G/P_d$, and gives
a canonical affine neighborhood for the point
$\tau(=\tau\,P_{d})$; further,the point $\tau\,P_{d}$ is
identified with the origin. The affine co-ordinates in $\tau
B^-\tau^{-1}\cdot \tau\,P_{d}$ may be indexed as
$\{x_\gamma,\gamma\in\tau(R^-\,\setminus\,R^-_{P_{d}})\}$ (here,
$R^-_{P_{d}}$ denotes the set of negative roots of $P_d$). We
recall the following two well known facts:

\vs.1cm\ni\textbf{Fact 1:} For any $f\in H^0(G/P_d,L(\omega_d))$,
the evaluations of ${\frac{\partial f}{\partial x_\gamma}}$ and
$X_\gamma f$ at $\tau(=\tau P_d)$ coincide, $X_\gamma$ being the
element in the Chevalley basis of $Lie\,G$ (the Lie algebra of
$G$), associated to $\gamma$.

\vs.1cm\ni\textbf{Fact 2:} For $f\in H^0(G/P_d,L(\omega_d))$, we
have that $ord_{\tau}\,f$ is the degree of the leading form (i.e.,
form of smallest degree) in the local polynomial expression for
$f$ at $\tau$.

In the sequel, for $f\in H^0(G/P_d,L(\omega_d))$,  we shall denote
the leading form in the polynomial expression for $f$ at $\tau$ by
$LF(f)$. Thus we are reduced to determining
$LF(p_{w_0^{(d)}}),1\le d\le l$.

\subsection{Determination of $LF(p_{w_0^{(d)}})$:}\label{compute}
Toward the determination of $LF(p_{w_0^{(d)}}),1\le d\le l$, we
first look for $\gamma_i$'s in $\tau(R^-\,\setminus\,R^-_{P_{d}})$
such that $X_{\gamma_1}^{n_1}\cdots
X_{\gamma_r}^{n_r}p_{w_0^{(d)}}$, $n_i\ge 1$, is a non-zero
multiple of $p_\tau$ (note that any monomial in the local
expression for $p_{w_0^{(d)}}$ arises from such a collection of
$\gamma_i$'s and $n_i$'s, in view of Fact 1; also note that a
$\gamma_i$ could repeat itself one or more times in
$X_{\gamma_1}^{n_1}\cdots X_{\gamma_r}^{n_r}p_{w_0^{(d)}}$).
Consider such an equality:
\[X_{\gamma_1}^{n_1}\cdots X_{\gamma_r}^{n_r}p_{w_0^{(d)}}=
cp_\tau, c\in k^* .\] Weight considerations imply
\[{\underset{1\le j\le r}{\sum}}n_j\gamma_j+i(\omega_d)=-\tau(\omega_d),
\] $i$ being the Weyl involution. Writing
$\gamma_j=-\tau(\beta_j)$, for a unique $\beta_j\in
R^+\,\setminus\,R^+_{P_d}$, we obtain
\[\tau(\omega_d)+i(\omega_d)={\underset{1\le j\le
r}{\sum}}n_j\tau(\beta_j),\] i.e.,
\[\omega_d+\tau(i(\omega_d))={\underset{1\le j\le
r}{\sum}}n_j\beta_j\] (note that $\tau=\tau^{-1}$).

 Also, using the facts that $p_\tau=c\tau p_{e^{(d)}}$
 (for some non-zero scalar $c$), and
 $X_{\gamma_i}=\tau X_{\beta_i}\tau^{-1}$, we obtain that
 $X_{\gamma_1}^{n_1}\cdots X_{\gamma_r}^{n_r}p_{w_0^{(d)}}$ is a
 non-zero scalar multiple of $p_\tau$ if and only if
 $X_{-\beta_1}^{n_1}\cdots X_{-\beta_r}^{n_r}p_{\tau w_0^{(d)}}$ is a
 non-zero scalar multiple of $p_{e^{(d)}}$. Thus,
\[ord_\tau p_{w_0^{(d)}}=ord_e p_{\tau w_0^{(d)}}.\] Hence,
we obtain that $ord_e p_{\tau w_0^{(d)}}$ equals
min$\,\{{\underset{1\le j\le r}{\sum}}n_j\}$ such that there exist

\ni $\{\beta_1,\cdots,\beta_r;n_1,\cdots,n_r,\ \beta_j\in
R^+\,\setminus\,R^+_{P_d},\,n_j\ge 1\}$ with
$X_{-\beta_1}^{n_1}\cdots X_{-\beta_r}^{n_r}p_{\tau w_0^{(d)}}$
being a non-zero scalar multiple of $p_{e^{(d)}}$.

 Thus in the following sections, for each minuscule $G/P$, we
 carry out Steps 1 \& 2 below. Also,  in view
of the results in \cite{mp,lrs}, we shall first carry out (Steps 1
\& 2 below) for the following minuscule $G/P$'s:

 I. $G$ of Type \textbf{D}, $P=P_1$

 II. $G$ of type \textbf{E}$_6$, $P=P_1,P_6$

 III. $G$ of type \textbf{E}$_7$, $P=P_7$.

 \begin{remark}
In view of the fact that for $G$ of type \textbf{E}$_6$,
$G/P_1\cong G/P_6$, we will restrict our attention to just $G/P_1$
when $G$ is of type \textbf{E}$_6$.
 \end{remark}

\begin{remark}
We need not consider $Sp(2n)/P_1$ (which is minuscule), since it
is isomorphic to $\mathbb{P}^{2n-1}$; further, as is easily seen,
$\mathbb{P}^N\times\mathbb{P}^N$ has a splitting which compatibly
splits the diagonal with maximum multiplicity (one may also deduce
this from \cite{mp} by identifying $\mathbb{P}^N$ with the
Grassmannian of $1$-dimensional subspaces of $k^{N+1}$). It should
be remarked (as observed in \cite{mp}) that for $G=Sp(2n)$,
$\sigma$ (as above) does not have maximum multiplicity along
$P_1/B$.
\end{remark}

 \vs.1cm\ni \textbf{Step 1:} For each $1\le d\le l$, we find the
 expression ${\underset{1\le j\le
l}{\sum}}c_j\alpha_j,\,c_j\in \mathbb{Z}^+$  for
$\omega_d+\tau(i(\omega_d))$ as a non-negative integral linear
combination of simple roots; in fact, as will be seen, we have
that $c_j\ne 0, \forall j$.

 \vs.1cm\ni \textbf{Step 2:} We show that min$\,\{{\underset{1\le j\le
 r}{\sum}}n_j\}$ (with notation as above) is given as follows:

 \textbf{D}$_n$, \textbf{E}$_6$: min$\,\{{\underset{1\le j\le
 r}{\sum}}n_j\}=c_1$

 \textbf{E}$_7$: min$\,\{{\underset{1\le j\le
 r}{\sum}}n_j\}=c_7$

 \ni Towards proving this, we observe that in \textbf{D}$_n$,
 \textbf{E}$_6$, coefficient of $\alpha_1$ in any positive root is
 less than or equal to one, while in \textbf{E}$_7$, the
 coefficient of $\alpha_7$ in any positive root is
 less than or equal to one (see \cite{bou}). Hence for any
 collection

 \ni $\{\beta_1,\cdots,\beta_r;n_1,\cdots,n_r,\ n_j\ge 1\}$
 as above, we have \[{\underset{1\le j\le
 r}{\sum}}n_j\ge\begin{cases} c_1,&\mathrm{\ if\ type\ } \mathbf{D}_n \mathrm{\ or\ }
 {\mathbf{E}}_6\\
 c_7, & \mathrm{\ if\ type\ }
 {\mathbf{E}}_7
\end{cases}\]
We then exhibit a collection $\{\beta_1,\cdots,\beta_r\}$,
$\beta_j\in R^+\,\setminus\,R^+_{P_d},1\le j\le r$ such that

 (a) $\omega_d+\tau(i(\omega_d))={\underset{1\le j\le
r}{\sum}}\beta_j$

(b) The reflections $s_{\beta_j}$'s (and hence the Chevalley basis
elements $X_{-\beta_j}$'s) mutually commute.

(c) For any subset $\{\delta_1,\cdots,\delta_s\}$ of
$\{\beta_1,...,\beta_{r}\}$, $X_{\delta_1}\cdots
X_{\delta_s}p_{\tau w_0^{(d)}}$ is an extremal weight vector (in
$H^0(G/P_d,L(\omega_d))$), and $X_{\beta_1}^{n_1}\cdots
X_{\beta_{r}}^{n_r}p_{\tau w_0^{(d)}}$ is a lowest weight vector (i.e.,
a non-zero scalar multiple of $p_{e^{(d)}}$).

(d) \[{\underset{1\le j\le
 r}{\sum}}n_j=\begin{cases} c_1,&\mathrm{\ if\ type\ } \mathbf{D}_n \mathrm{\ or\ }
 {\mathbf{E}}_6\\
 c_7, & \mathrm{\ if\ type\ }
 {\mathbf{E}}_7
\end{cases}\]

(e) We then conclude (by the foregoing discussion) that
$$ord_{\tau}\,p_{w_0^{(d)}}=\begin{cases} c_1,&\mathrm{\ if\ type\ } \mathbf{D}_n \mathrm{\ or\ }
 {\mathbf{E}}_6\\
 c_7, & \mathrm{\ if\ type\ }
 {\mathbf{E}}_7
\end{cases}$$

\begin{rem}\label{path}
 Thus we obtain a nice realization for $ord_{\tau}\,p_{w_0^{(d)}}$
 (the order of vanishing along $P/B$ of $p_{w_0^{(d)}}$) as being the length of
the shortest path through extremal weights in the weight lattice
connecting the highest weight (namely, $i(\omega_d)$ in
$H^0(G/B,L(\omega_d))$ and the extremal weight $-\tau(\omega_d)$.
\end{rem}

\begin{rem}
For the sake of completeness, we have given the details for the
remaining $G/P$'s in \S \ref{last}.
\end{rem}

 For the convenience of notation, we make the following
\begin{defn}\label{md}
Define $m_d$ to be $ord_{e}\,p_{\tau
w_0^{(d)}}(=ord_{\tau}\,p_{w_0^{(d)}})$
\end{defn}

 %\vs.1cm\ni \textbf{Step 3}

 %\vs.1cm\ni \textbf{Step 4}

 We shall treat the cases I,II,III above, respectively
  in the following three sections. In the following sections,
  we will be repeatedly using the
  following:

  \vs.2cm\ni\textbf{Fact 3:} Suppose $p_\theta$ is an extremal
  weight vector in $H^0(G/P_d,L(\omega_d))$ of weight
  $\chi(=-\theta(\omega_d))$, and
  $\beta\in R$ such that
  $(\chi,\beta^*)=r$, for some positive integer $r$. Then
  $X^r_{-\beta}p_\theta$ is a non-zero scalar multiple of the
  extremal weight vector $p_{s_\beta\theta}$
  (here, $(,)$ is a $W$-invariant scalar product on the
  weight lattice, and
  $(\chi,\beta^*)={\frac{2(\chi,\beta)}{(\beta,\beta)}}$).

  The above fact follows from $sl(2)$-theory (note that $p_\theta$ is a
  highest weight vector for the Borel sub group $\theta B^-\theta^{-1}, B^-$
  being the Borel subgroup opposite to $B$).

\section{The minuscule $SO(2n)/P_1$}\label{dn}
Let the characteristic of $k$ be different from $2$. Let $V=k^{2n}$
together with a non-degenerate symmetric bilinear form $(\cdot
,\cdot )$. Taking the matrix of the form $(\cdot ,\cdot )$ (with
respect to the standard basis $\{e_1,\ldots ,e_{2n} \}$ of $V$) to
be $E$, the anti-diagonal ($1,\ldots ,1$) of size $2n \times 2n $.
We may realize $G=SO(V)$ as the fixed point set $SL(V)^\sigma $,
where $\sigma :SL(V) \rightarrow SL(V) $ is given by $\sigma
(A)=E(^tA)^{-1}E$. Set $H=SL(V)$.

Denoting by $T_H$ (resp.\ $B_H$) the maximal torus in $H$
consisting of diagonal matrices (resp.\ the Borel subgroup in $H$
consisting of upper triangular matrices) we see easily that $T_H,
B_H$ are stable under $\sigma$. We set $T_G={T_H}^{\sigma} , \
B_G={B_H}^{\sigma}.$ Then it is well known that $T_G$ is a maximal
torus in $G$ and $B_G$ is a Borel subgroup in $G$. We have a
natural identification of the Weyl group $W$ of $G$ as a subgroup
of $S_{2n}:$
\[W=\{(a_1 \cdots a_{2n}) \in S_{2n} \mid a_i=2n+1-a_{2n+1-i},\ 1
\leq i \leq 2n,\emph{ and }m_w \emph{ is even} \}\]  where
$m_w=\#\{i\le n\,|\,a_i>n\}$. Thus $w=(a_1\cdots a_{2n}) \in W_G$
is known once $(a_1 \cdots a_n)$ is known. In the sequel, we shall
denote such a $w$ by just $(a_1 \cdots a_n)$; also, for $1\le i\le
2n$, we shall denote $2n+1-i$ by $i'$.

For details see \cite{lr}.

Let $P=P_{\alpha_1}$. We preserve the notation from the previous
section; in particular, we denote the element of largest length in
$W_P$ by $\tau$.  We have
$$ \tau=\begin{cases}(12'3'\cdots (n-1)'n),
&\mathrm{if}\ n\mathrm{\ is\ even}\\
                 (12'3'\cdots (n-1)'n'),
&\mathrm{if}\ n\mathrm{\ is\ odd}
\end{cases}
$$

\vs.2cm\ni\textbf{Steps 1 \& 2 of \S \ref{compute}:} As in
\cite{bou}, we shall denote by $\epsilon_j,1\le j\le n$, the
restriction to $T_G$ of the character of $T_H$, sending a diagonal
matrix $diag\{t_1,\cdots,t_n\}$ to $t_j$.

 \ni\textbf{Case 1:} Let
$d\le n-2$. Then $\omega_d=\epsilon_1+\cdots+\epsilon_d$ (cf.
\cite{bou}). We have, $i(\omega_d)=\omega_d$, and
$\tau(\omega_d)=\epsilon_1+\epsilon_{2'}+\cdots+\epsilon_{d'}$.
Hence,
\[\omega_d+\tau(i(\omega_d))=2\epsilon_1=
2(\alpha_1+\cdots+\alpha_{n-2})+\alpha_{n-1}+\alpha_{n}\] (note
that an element in $T_G$ is of the form
$diag\{t_1,\cdots,t_n,t_n^{-1},\cdots,t_1^{-1}\}$, and hence
$\epsilon_{j'}=-\epsilon_j$(we follow \cite{bou} for denoting the
simple roots)). We let $\{\beta_1,\beta_2\}\subset
R^+\,\setminus\,R^+_{P_d}$ be any (unordered) pair of the form
$\{\epsilon_1-\epsilon_j, \epsilon_1+\epsilon_j,d+1\le j\le n\}$.
Clearly, $s_{\beta_1}, s_{\beta_2}$ commute (since,
$(\beta_1,\beta_2^*)=0$), and
$\omega_d+\tau(i(\omega_d))=\beta_1+\beta_2$. Also,
\[\begin{gathered}{(-\tau w_0^{(d)}(\omega_d)},\beta_j^*)=(\tau
(\omega_d),\beta_j^*)=1,j=1,2 ;\\  (-\tau
w_0^{(d)}(\omega_d)-\beta_j,\beta_m^*)=(\tau
(\omega_d)-\beta_j,\beta_m^*)=1,\ j,m\in\{1,2\},\ \mathrm{and\
}j,m\
\mathrm{distinct};\\
-\tau w_0^{(d)}(\omega_d)-\beta_1-\beta_2=\tau
(\omega_d)-\beta_1-\beta_2=-\omega_d=
-(\epsilon_1+\cdots+\epsilon_d).
\end{gathered}
\] From this, (a)-(c) in Step 2 of \S \ref{compute} follow
for the above choice of $\{\beta_1,\beta_2\}$; (d) in Step 2 is
obvious. Hence $m_d=2,1\le d\le n-2$ (recall $m_d$ from Definition
\ref{md}).

\vs.2cm\ni\textbf{Case 2:} $d=n-1$. We have, $i(\omega_{n-1})$
equals $\omega_{n-1}$ or $\omega_{n}$, according as $n$ is even or
odd.

\ni If $n$ is even, then
$\tau(i(\omega_{n-1}))=\tau(\omega_{n-1})={\frac{1}{2}}
(\epsilon_1+\epsilon_{2'}+\cdots+\epsilon_{(n-1)'}-\epsilon_n)$.

\ni If $n$ is odd, then
$\tau(i(\omega_{n-1}))=\tau(\omega_{n})={\frac{1}{2}}
(\epsilon_1+\epsilon_{2'}+\cdots+\epsilon_{(n-1)'}+\epsilon_{n'})$

\ni $={\frac{1}{2}}
(\epsilon_1+\epsilon_{2'}+\cdots+\epsilon_{(n-1)'}-\epsilon_n)$.

Thus in either case, $\tau(i(\omega_{n-1}))={\frac{1}{2}}
(\epsilon_1+\epsilon_{2'}+\cdots+\epsilon_{(n-1)'}-\epsilon_n)$.
Hence
\[\omega_{n-1}+\tau(i(\omega_{n-1}))=\epsilon_1-\epsilon_n=
\alpha_1+\cdots+\alpha_{n-2}+\alpha_{n-1}\]
 which is clearly a root in $R^+\,\setminus\,R^+_{P_{n-1}}$. Hence
 taking $\beta_1$ to be $\epsilon_1-\epsilon_n$, we find that
 $\{\beta_1\}$ (trivially) satisfies (a)-(c) in Step 2 of \S \ref{compute}
  follow; (d) in Step 2 is obvious. Hence
$m_{n-1}=1$.

\vs.2cm\ni\textbf{Case 3:} $d=n$. Proceeding as in Case 2, we
have,
\[\omega_{n}+\tau(i(\omega_{n}))=\epsilon_1+\epsilon_n=\alpha_1+\cdots+\alpha_{n-2}+\alpha_{n}\]
 which is clearly a root in $R^+\,\setminus\,R^+_{P_{n}}$. As in
 case 2, we conclude that $m_n=1$.

 \begin{thm}\label{main1}
The LMP conjecture holds for the minuscule $G/P$, $G,P$ being as
above.
 \end{thm}
 \begin{proof}
From \S \ref{steps} (see ``Main reduction" in that section), we
just need to show that ${\underset{1\le d\le n}{\sum}}ord_\tau
p_{w^{(d)}_0}$ equals codim$_{G/B}P/B$. From the above
computations, and the discussion in \S \ref{follow},
 we have ${\underset{1\le d\le
n}{\sum}}ord_\tau p_{w^{(d)}_0}$ equals ${\underset{1\le d\le
n}{\sum}}m_d=2n-2$ which is precisely codim$_{G/B}P/B$.
 \end{proof}
%%%%%%%%%%%%%%%%%%%%%%%%%%%%%%%%%%%%%%%%%%%%%%%%%%%%%

\section{Exceptional Group E$_6$}\label{e6}

Let $G$ be simple of type $E_6$. Let $P=P_1$. As in the previous
sections, let $\tau$ be the unique element of largest length in
$W_P$.

\vs.2cm\ni\textbf{Step 1 \& 2 of \S \ref{compute}:} Note that
$\tau$ is the unique element of largest length inside the Weyl
group of type D$_5$; we have that D$_5$ sits inside of E$_6$ as
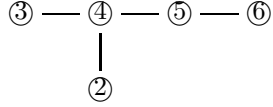
\[\xymatrix@-10pt{
 \mbox{\textcircled{3}}\ar@{-}[r] & \mbox{\textcircled{4}}\ar@{-}[r]\ar@{-}[d] & \mbox{\textcircled{5}}\ar@{-}[r] & \mbox{\textcircled{6}} \\
 & \mbox{\textcircled{2}} & &
}\] Thus, for $2\leq j\leq 6$, we have $\tau
(\alpha_j)=-i(\alpha_j )$, $i$ being the Weyl involution of D$_5$;
in this case, we have
\[\tau(\alpha_2)=-\alpha_3,\tau(\alpha_3)=-\alpha_2,\tau(\alpha_j)=-\alpha_j
,j=4,5,6 . \leqno{(*)}\] Thus, using the Tables in \cite{bou} , to
find $\tau(i(\omega_d)),1\le d\le 6$, as a linear sum (with
rational coefficients) of the simple roots, it remains to find
$\tau (\alpha_1)$.

Let $\tau(\alpha_1)={\underset{1\le j\le 6}{\sum}}
a_j\alpha_j,a_j\in\mathbb{Z}$. Since $\alpha_1\not\in R_P$ (the
root system of $P$), we have that $\tau(\alpha_1)\not\in R_P$.
Hence $a_1\ne 0$; further, $\tau(\alpha_1)\in R^+$ (since,
clearly, $l(\tau s_{\alpha_1})=l(\tau)+1$). Hence, $a_1>0$; in
fact, we have, $a_1=1$ (since any positive root in the root system
of E$_6$ has an $\alpha_1$ coefficient $\le 1$). Using $(*)$ above,
and the following linear system, we determine the remaining
$a_j$'s:
\begin{eqnarray*}
 2a_2-a_4 & = & \left<\tau(\alpha_1) , \alpha_2^\ast\right> = \left<\alpha_1,\tau(\alpha_2^\ast )\right> = \left<\alpha_1,-\alpha_3^\ast\right>=1\\
 2a_3-a_1-a_4 & = & \left<\tau(\alpha_1) , \alpha_3^\ast\right> = \left<\alpha_1,\tau(\alpha_3^\ast )\right> = \left<\alpha_1,-\alpha_2^\ast\right>=0\\
 2a_4-a_2-a_3-a_5 & = & \left<\tau(\alpha_1) , \alpha_4^\ast\right> = \left<\alpha_1,\tau(\alpha_4^\ast )\right> = \left<\alpha_1,-\alpha_4^\ast\right>=0\\
 2a_5-a_4-a_6 & = & \left<\tau(\alpha_1) , \alpha_5^\ast\right> = \left<\alpha_1,\tau(\alpha_5^\ast )\right> = \left<\alpha_1,-\alpha_5^\ast\right>=0\\
 2a_6-a_5 & = & \left<\tau(\alpha_1) , \alpha_6^\ast\right> = \left<\alpha_1,\tau(\alpha_6^\ast )\right> = \left<\alpha_1,-\alpha_6^\ast\right>=0
 \end{eqnarray*}
Either one may just solve the above linear system or use the
properties of the root system of type \textbf{E}$_6$ to
 quickly solve for $a_j$'s. For instance, we have, $a_6\ne 0$; for,
  $a_6= 0$ would imply (working with the last equation and up) that
 $a_4= 0$ which in turn would imply (in view of the first equation) that
  $a_2={\frac{1}{2}}$, not possible. Hence $a_6\ne 0$, and
  in fact equals $1$ (for the same reasons as in
   concluding that $a_1=1$).
   Once again working backward in the linear system, $a_5=2,a_4=3$; hence
   from the first equation, we obtain, $a_2=2$. Now the second
    equation implies that $a_3=2$. Thus we
    obtain $$\tau(\alpha_1 ) = \alpha_1+2\alpha_2 +
2\alpha_3+3\alpha_4+2\alpha_5+\alpha_6 =
\begin{pmatrix} 1 & 2 & 3 & 2 & 1 \\ & & 2  & &\end{pmatrix}.$$

For $d\in\{1,\ldots ,6\}$, we shall now describe
$\{\beta_1,\cdots,\beta_r\mid \beta_i\in R^+\setminus R^+_{P_d}\}$
which satisfies the conditions (a)-(d) in Step 2 of \S
\ref{compute}.

For convenience, we list the fundamental weights here:
\[\omega_1 =
\frac{1}{3}\begin{pmatrix} 4 & 5 & 6 & 4 & 2 \\ & & 3 & &
\end{pmatrix},\ \omega_2 = \begin{pmatrix} 1 & 2 & 3 & 2 & 1 \\ & & 2 & & \end{pmatrix},\ \omega_3 = \frac{1}{3}\begin{pmatrix} 5 & 10 & 12 & 8 & 4 \\ & & 6 & & \end{pmatrix} \]
\[\omega_4 = \begin{pmatrix} 2 & 4 & 6 & 4 & 2  \\ & & 3 & & \end{pmatrix},\ \omega_5 = \frac{1}{3}\begin{pmatrix} 4 & 8 & 12 & 10 & 5 \\ & & 6 & & \end{pmatrix},\ \omega_6 = \frac{1}{3}\begin{pmatrix} 2 & 4 & 6 & 5 & 4 \\ & & 3 & & \end{pmatrix}\]
\vspace{.2 in}

\vs.2cm\ni \textbf{Case 1:} $d=1$.

\ni We have, $i(\omega_1 ) = \omega_6 $. Hence using (from above),
the expression for $\omega_6$ as a (rational) sum of simple roots,
and the expressions for $\tau(\alpha_j), j=1,\cdots,6$, we obtain
\[\tau (i(\omega_1 ))+\omega_1 =\begin{pmatrix} 2 & 2 & 2 & 1 & 0
\\ & & 1 & &
\end{pmatrix} .\] We let $\{\beta_1,\beta_2\}$
 be the unordered pair of roots:
\[\begin{pmatrix} 1 & 1 & 1 & 0 & 0
\\ & & 0 & &
\end{pmatrix}, \begin{pmatrix} 1 & 1 & 1 & 1 & 0
\\ & & 1 & &
\end{pmatrix} .\] Clearly $\beta_1,\beta_2$ are in
$R^+\setminus R^+_{P_1}$, and the reflections
$s_{\beta_1},s_{\beta_2}$ commute (since $\beta_1+\beta_2$ is not
a root). Further, $\tau (i(\omega_1 ))+\omega_1= \beta_1+\beta_2$
. Also,
\[\begin{gathered}
 \left<-\tau w^{(1)}_0(\omega_1),\beta_j^*\right>=
 \left<\tau (i(\omega_1)),\beta_j^*\right>=
\left<\beta_1+\beta_2-\omega_1,\beta_j^*\right>=1,\ j=1,2;\\
\left<-\tau
w^{(1)}_0(\omega_1)-\beta_l,\beta_j^*\right>=\left<\beta_j-\omega_1,\beta_j^*\right>=1,
\ j,l\in\{1,2\},\mathrm{\ and\ } j,l \mathrm{\ distinct};\\
-\tau w_0^{(1)}(\omega_1)-\beta_1-\beta_2=-\omega_1 .
\end{gathered}\] From this, (a)-(c) in Step 2 of \S \ref{compute} follow
for the above choice of $\{\beta_1,\beta_2\}$; (d) in Step 2 is
obvious. Hence $m_1=2$ (cf. Definition \ref{md}).

\vs.2cm\ni \textbf{Case 2:} $d=2$.

\ni We have $i(\omega_2 )=\omega_2$. As in case 1, using the
expression for $\omega_2$ as a (rational) sum of simple roots, and
the expressions for $\tau(\alpha_j), j=1,\cdots,6$, we obtain
 $$\tau(\omega_2 )+\omega_2 = \begin{pmatrix} 2 & 2 & 3 & 2 & 1
\\ & & 2 & & \end{pmatrix} .$$
We let $\{\beta_1,\beta_2\}$  be the unordered pair of roots:
\[ \begin{pmatrix} 1 & 1 & 2 & 2 & 1 \\ & & 1 & & \end{pmatrix},\
 \begin{pmatrix} 1 & 1 & 1 & 0 & 0 \\ & & 1 & & \end{pmatrix} .\]
Clearly $\beta_1,\beta_2$ are in $R^+\setminus R^+_{P_2}$, and the
reflections $s_{\beta_1},s_{\beta_2}$ commute (since
$\beta_1+\beta_2$ is not a root). Further, $\tau(\omega_2
)+\omega_2 = \beta_1+\beta_2$.  We have
\[\begin{gathered}
 \left<-\tau w^{(2)}_0(\omega_2),\beta_j^*\right>=
 \left<\tau (\omega_2),\beta_j^*\right>=
\left<\beta_1+\beta_2-\omega_2,\beta_j^*\right>=1,\ j=1,2;\\
\left<-\tau
w^{(2)}_0(\omega_2)-\beta_l,\beta_j^*\right>=\left<\beta_j-\omega_2,\beta_j^*\right>=1,
\ j,l\in\{1,2\},\mathrm{\ and\ } j,l \mathrm{\ distinct};\\
-\tau w_0^{(2)}(\omega_2)-\beta_1-\beta_2=-\omega_2.
\end{gathered}\] As in case 1, (a)-(c) in Step 2 of \S \ref{compute} follow
for the above choice of $\{\beta_1,\beta_2\}$; (d) in Step 2 is
also clear. Hence $m_2=2$.

\vs.2cm The discussion in the remaining cases are similar; in each
case we will just give the expression for
$\tau(i(\omega_d))+\omega_d$ as an element in the root lattice,
and the choice of $\{\beta_1,\cdots,\beta_r\}$ in
$R^+\,\setminus\,R^+_{P_d}$ which satisfy the conditions (a)-(d)
in Step 2 of \S \ref{compute}. Then deduce the value of $m_d$.

\vs.2cm\ni \textbf{Case 3:} $d=3$.

\ni We have $i(\omega_3)=\omega_5$. Further,
$$\tau(\omega_5 )+\omega_3 = \begin{pmatrix} 3 & 4 & 4 & 2 & 1
\\ & & 2 & & \end{pmatrix}.$$  We let $\{\beta_1,\beta_2,\beta_3\}$
be the unordered triple of roots:
\[ \begin{pmatrix} 1 & 1 & 1 & 1 & 1 \\ & & 0 & & \end{pmatrix},\
 \begin{pmatrix} 1 & 2 & 2 & 1 & 0 \\ & & 1 & & \end{pmatrix},\
  \begin{pmatrix} 1 & 1 & 1 & 0 & 0 \\ & & 1 & & \end{pmatrix}.\]
Then we have $\tau(\omega_3 )+\omega_3 = \beta_1
+\beta_2+\beta_3$. Reasoning as in case 1, we conclude
$m_3=3$.

\vs.2cm\ni \textbf{Case 4:} $d=4$.

\ni We have $i(\omega_4)=\omega_4$. Further, $$\tau (\omega_4
)+\omega_4 =
\begin{pmatrix} 4 & 5 & 6 & 4 & 2 \\ & & 3 & & \end{pmatrix}.$$  We
let $\{\beta_1,\beta_2,\beta_3,\beta_4\}$ be the unordered
quadruple of roots:
\[ \begin{pmatrix} 1 & 2 & 2 & 2 & 1 \\ & & 1 & & \end{pmatrix},\
 \begin{pmatrix} 1 & 1 & 2 & 1 & 0 \\ & & 1 & & \end{pmatrix}
,\] \[ \begin{pmatrix} 1 & 1 & 1 & 1 & 1 \\ & & 1 & &
\end{pmatrix},\ \begin{pmatrix} 1 & 1 & 1 & 0 & 0 \\ & &
0 & & \end{pmatrix}.\] Then we have $\tau(\omega_4 )+\omega_4 =
\beta_1 +\beta_2+\beta_3+\beta_4$.  Reasoning as in case 1, we
conclude $m_4=4$.

\vs.2cm\ni \textbf{Case 5:} $d=5$.

\ni We have $i(\omega_5)=\omega_3$. Further
$$\tau(\omega_3)+\omega_5 = \begin{pmatrix} 3 & 4 & 5 & 4 & 2
\\ & & 2 & & \end{pmatrix}.$$ We let $\{\beta_1,\beta_2,\beta_3\}$
be the unordered triple of roots:
\[ \begin{pmatrix} 1 & 2 & 3 & 2 & 1 \\ & & 1 & & \end{pmatrix},\
 \begin{pmatrix} 1 & 1 & 1 & 1 & 0 \\ & & 0 & & \end{pmatrix}
,\  \begin{pmatrix} 1 & 1 & 1 & 1 & 1 \\ & & 1 & &
\end{pmatrix}.\] We have, $\tau(\omega_3)+\omega_3 =
\beta_1+\beta_2+\beta_3$.  We proceed as in case 1, and conclude
$m_5=3$.

\vs.2cm\ni \textbf{Case 6:} $d=6$.

We have, $i(\omega_6)=\omega_1$.  Further,
$\tau(\omega_1)+\omega_6 =
\begin{pmatrix} 2 & 3 & 4 & 3 & 2 \\ & & 2 & & \end{pmatrix}$.
We let $\{\beta_1,\beta_2\}$ be the unordered pair of roots:
\[ \begin{pmatrix} 1 & 1 & 1 & 1 & 1 \\ & & 0 & & \end{pmatrix},\
 \begin{pmatrix} 1 & 2 & 3 & 2 & 1 \\ & & 2 & & \end{pmatrix}.\]
We have $\tau(\omega_1)+\omega_6 = \beta_1+\beta_2.$  Proceeding
as in case 1, we conclude $m_6=2$.

\begin{thm}\label{main2}
The LMP conjecture holds for the minuscule $G/P$, $G,P$ being as
above.
 \end{thm}
 \begin{proof}
 Proceeding as in the proof of Theorem \ref{main1}, we
just need to show that ${\underset{1\le d\le 6}{\sum}}ord_\tau
p_{w^{(d)}_0}$ equals codim$_{G/B}P/B$. From the above
computations, we have ${\underset{1\le d\le 6}{\sum}}ord_\tau
p_{w^{(d)}_0}$ equals $16 $ which is precisely codim$_{G/B}P/B$.
 \end{proof}

\section{ Exceptional Group E$_7$}\label{e7}

Let $G$ be simple of type E$_7$. Let $P$ be the maximal parabolic
subgroup associated to the fundamental weight $\omega_7$ (the only
minuscule weight in E$_7$).  We preserve the notation of the
previous sections; in particular, $\tau$ will denote the unique
element of largest length in $W_P$.

\vs.2cm\ni\textbf{Step 1 \& 2 of \S \ref{compute}:} Note that
$\tau$ is the unique element of largest length inside the Weyl
group of type \textbf{E}$_6$; \textbf{E}$_6$ sits inside
\textbf{E}$_7$ in the natural way:
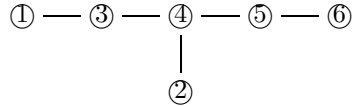
\[\xymatrix@-10pt{
\mbox{\textcircled{1}} \ar@{-}[r] & \mbox{\textcircled{3}}\ar@{-}[r] & \mbox{\textcircled{4}}\ar@{-}[r]\ar@{-}[d] & \mbox{\textcircled{5}}\ar@{-}[r] & \mbox{\textcircled{6}} \\
 & & \mbox{\textcircled{2}} & &
}\] Thus, for $\alpha_i,\,1\leq i\leq 6$, we have, $\tau (\alpha_i
)=-i(\alpha_i )$, where $i$ is the Weyl involution on E$_6$. To be
very precise, we have,
\[\tau(\alpha_1)=-\alpha_6,\,\tau(\alpha_3)=-\alpha_5,\,\tau(\alpha_i)
=-\alpha_i,i=2,4.\leqno{(*)}\] Thus, using \cite{bou}, to find
$\tau(i(\omega_d)),1\le d\le 7$, as a linear sum (with rational
coefficients) of the simple roots, it remains to find
$\tau(\alpha_7)$. Towards computing $\tau(\alpha_7)$, we proceed
as in \S \ref{e6}. Let $\tau(\alpha_7)=\sum_{i=1}^7 a_i\alpha_i$,
where $a_i\in \mathbb{Z}$. Since $\alpha_7\not\in R_P$ (the root
system of $P$), we have that $\tau(\alpha_7)\not\in R_P$. Hence
$a_7\ne 0$; further, $\tau(\alpha_7)\in R^+$ (since, clearly,
$l(\tau s_{\alpha_7})=l(\tau)+1$). Hence, $a_7>0$; in fact, we
have, $a_7=1$ (since any positive root in the root system of E$_7$
has an $\alpha_7$ coefficient $\le 1$). Using $(*)$ above, and the
following linear system, we determine the remaining $a_j$'s:
\begin{eqnarray*}
 2a_1-a_3 & = & \left<\tau(\alpha_7) , \alpha_1^\ast\right> = \left<\alpha_7, \tau(\alpha_1^\ast )\right> = \left<\alpha_7,-\alpha_6^\ast\right>=1\\
 2a_2-a_4 & = & \left<\tau(\alpha_7) , \alpha_2^\ast\right> = \left<\alpha_7,\tau(\alpha_2^\ast )\right> = \left<\alpha_7,-\alpha_2^\ast\right>=0\\
 2a_3-a_4-a_1 & = & \left<\tau(\alpha_7) , \alpha_3^\ast\right> = \left<\alpha_7,\tau(\alpha_3^\ast )\right> = \left<\alpha_7,-\alpha_5^\ast\right>=0\\
 2a_4-a_2-a_3-a_5 & = & \left<\tau(\alpha_7) , \alpha_4^\ast\right> = \left<\alpha_7,\tau(\alpha_4^\ast )\right> = \left<\alpha_7,-\alpha_4^\ast\right>=0\\
 2a_5-a_4-a_6 & = & \left<\tau(\alpha_7) , \alpha_5^\ast\right> = \left<\alpha_7,\tau(\alpha_5^\ast )\right> = \left<\alpha_7,-\alpha_3^\ast\right>=0\\
 2a_6-a_5-a_7 & = & \left<\tau(\alpha_7) , \alpha_6^\ast\right> = \left<\alpha_7,\tau(\alpha_6^\ast )\right> = \left<\alpha_7,-\alpha_1^\ast\right>=0
 \end{eqnarray*} The fact that $a_7=1$ together with the last
 equation implies $a_5\ne 0$
 (and hence $a_6\ne 0$, again from the last equation;
 note that all $a_i\in \mathbb{Z}^+$). Similarly,
 from the first equation, we conclude $a_3\ne 0$
 (and hence $a_1\ne 0$). From the first and third equations, we
 conclude $a_4\ne 0$
 (and hence $a_2\ne 0$, in view of the second equation). Thus, all
 $a_i$'s are non-zero. The fifth equation implies that $a_4,a_6$
 are of the same parity, and are in fact both even
  (in view of the second equation); hence $a_6=2$. Now working
  with the last equation and up, we obtain
  \[a_5=3,a_4=4,a_2=2,a_3=3,a_1=2.\] Thus
   \[\tau(\alpha_7) =
2\alpha_1 + 2\alpha_2 + 3\alpha_3 +4\alpha_4 + 3\alpha_5 +
2\alpha_6 +\alpha_7 = \begin{pmatrix} 2 & 3 & 4 & 3 & 2 & 1 \\ & &
2 & & & \end{pmatrix}.\]

We proceed as in \S \ref{e6}. Of course, the Weyl involution for
\textbf{E}$_7$ is just the identity map. For each maximal
parabolic subgroup $P_d,1\le d\le 7$, we will give the expression
for $\tau(i(\omega_d))+\omega_d (= \tau(\omega_d)+\omega_d)$ as an
element in the root lattice, and the choice of
$\{\beta_1,\cdots,\beta_r\}$ in $R^+\,\setminus\,R^+_{P_d}$ which
satisfy the conditions (a)-(d) in Step 2 of \S \ref{compute}. Then
deduce the value of $m_d$.

For convenience, we list the fundamental weights here:
\[\omega_1 = \begin{pmatrix} 2 & 3 & 4 & 3 & 2 & 1 \\ & & 2 & & &
\end{pmatrix}, \]
  \[\omega_2 = \frac{1}{2}\begin{pmatrix} 4 & 8 & 12 & 9 & 6 & 3 \\ & & 7 & & & \end{pmatrix} \
  ,\omega_3 = \begin{pmatrix} 3 & 6 & 8 & 6 & 4 & 2 \\ & & 4 & & &
  \end{pmatrix}, \]
\[\omega_4 = \begin{pmatrix} 4 & 8 & 12 & 9 & 6 & 3 \\ & & 6 & & & \end{pmatrix} \
  ,\omega_5 = \frac{1}{2}\begin{pmatrix} 6 & 12 & 18 & 15 & 10 & 5 \\ & & 9 & & & \end{pmatrix},\]
\[\omega_6 = \begin{pmatrix} 2 & 4 & 6 & 5 & 4 & 2 \\ & & 3 & & & \end{pmatrix}\
  ,\omega_7 = \frac{1}{2} \begin{pmatrix} 2 & 4 & 6 & 5 & 4 & 3 \\ & & 3 & & & \end{pmatrix}.\]

\vs.2cm\ni \textbf{Case 1:} $d=1$.

\ni We have  $\tau(\omega_1 )+\omega_1 = \begin{pmatrix}2 & 3 & 4
& 3 & 2 & 2 \\ & & 2 & & &\end{pmatrix}$.  We let
$\{\beta_1,\beta_2\}$ be the unordered pair of roots:
\[ \begin{pmatrix} 1 & 2 & 3 & 2 & 1 & 1 \\ & & 2 & & &
\end{pmatrix},\   \begin{pmatrix} 1 & 1 & 1 & 1 & 1 & 1 \\ & & 0 & & &\end{pmatrix}.\]
Then we have $\tau(\omega_1 )+\omega_1=\beta_1+\beta_2$, and
$m_1=2$.

\vs.2cm\ni \textbf{Case 2:} $d=2$.
\subsection*{$Q=P_2$}

We have $\tau(\omega_2)+\omega_2 = \begin{pmatrix} 2 & 4 & 6 & 5 &
4 & 3
\\ & & 3 & & & \end{pmatrix}$.  We let $\{\beta_1,\beta_2,\beta_3\}$
 be the unordered triple of roots:
\[  \begin{pmatrix} 0 & 1 & 2 & 2 & 1 & 1 \\ & & 1 & & &
\end{pmatrix},\  \begin{pmatrix} 1 & 1 & 1 & 1 & 1 & 1 & 1 \\
& & 1 & & & \end{pmatrix},\  \begin{pmatrix} 1 & 2 & 3 & 2 & 2 & 1
\\ & & 1 & & & \end{pmatrix}.\]
Then we have $\tau(\omega_2 )+\omega_2=\beta_1+\beta_2+\beta_3$,
and $m_2=3$.

\vs.2cm\ni \textbf{Case 3:} $d=3$.

We have $\tau(\omega_3 )=\begin{pmatrix} 0 & 0 & 0 & 0 & 1 & 2 \\
& & 0 & & & \end{pmatrix}$; thus $\tau(\omega_3)+\omega_3 =
\begin{pmatrix} 3 & 6 & 8 & 6 & 5 & 4 \\ & & 4 & &
&\end{pmatrix}$.  We let $\{\beta_i, i=1,..,4\}$ be the unordered
quadruple of roots:
\[ \begin{pmatrix} 1 & 2 & 3 & 2 & 2 & 1 \\ & & 1 & & &
\end{pmatrix},\  \begin{pmatrix} 1 & 2 & 2 & 2 & 1 & 1 \\
& & 1 & & & \end{pmatrix},\]
\[  \begin{pmatrix} 1 & 1 & 1 & 1 & 1 & 1\\ & & 1 & & &\end{pmatrix}
,\ \begin{pmatrix} 0 & 1 & 2 & 1 & 1 & 1\\ & & 1 & & &
\end{pmatrix}.\]
Then we have $\tau(\omega_3 )+\omega_3={\underset{1\le i\le
4}{\sum}}\beta_i$, and $m_3=4$.

\vs.2cm\ni \textbf{Case 4:} $d=4$.

We have $\tau(\omega_4 )+\omega_4=
\begin{pmatrix} 4 & 8 & 12 & 10 & 8 & 6\\ & & 6 & & &
\end{pmatrix}$.  We let $\{\beta_i, i=1,..,6\}$ be the
unordered $6$-tuple of roots:
\[ \begin{pmatrix} 1& 2 & 3 & 3 & 2 & 1\\ & & 1 & & &
\end{pmatrix},\  \begin{pmatrix} 1 & 2 & 2 & 2 & 2 & 1 \\
& & 1 & & & \end{pmatrix},\  \begin{pmatrix} 0 & 1 & 2 & 2 & 1 & 1
\\ & & 1 & & & \end{pmatrix},\]
\[ \begin{pmatrix} 1 & 2 & 2 & 1 & 1 & 1\\ & & 1 & & &
\end{pmatrix},\  \begin{pmatrix} 1 & 1 & 2 & 1 & 1 & 1\\
& & 1 & & & \end{pmatrix},\  \begin{pmatrix} 0 & 0 & 1 & 1 & 1 &
1\\ & & 1 & & & \end{pmatrix}.\] Then we have $\tau(\omega_4
)+\omega_4={\underset{1\le i\le 6}{\sum}}\beta_i$, and $m_4=6$.

\vs.2cm\ni \textbf{Case 5:} $d=5$.

We have  $\tau(\omega_5)+\omega_5 = \begin{pmatrix} 3 & 6 & 10 & 9
& 7 & 5 \\ & & 5 & & &
\end{pmatrix}$.  We let $\{\beta_i, i=1,..,5\}$ be the unordered
quintuple of roots:
\[  \begin{pmatrix} 1 & 2 & 3 & 3 & 2 & 1 \\ & & 1 & & &
\end{pmatrix},\  \begin{pmatrix} 0 & 1 & 2 & 2 & 1 & 1 \\
& & 1 & & & \end{pmatrix},\
\begin{pmatrix} 0 & 1 & 1 & 1 & 1 & 1
\\ & & 1 & & & \end{pmatrix},\]
\[ \begin{pmatrix} 1 & 1 & 2 & 2 & 2 & 1 \\ & & 1 & & &
\end{pmatrix},\  \begin{pmatrix} 1 & 1 & 2 & 1 & 1 & 1\\ & & 1 & & & \end{pmatrix}.\]
Then we have $\tau(\omega_5 )+\omega_5={\underset{1\le i\le
5}{\sum}}\beta_i$, and $m_5=5$.

\vs.2cm\ni \textbf{Case 6:} $d=6$.

We have $\tau(\omega_6)+\omega_6 =
\begin{pmatrix} 2 & 5 & 8 & 7 & 6 & 4 \\ & & 4 & & &
\end{pmatrix}$.  We let $\{\beta_i, i=1,..,4\}$ be the unordered
quadruple of roots:
\[ \begin{pmatrix} 1 & 2 & 3 & 2 & 2 & 1\\ & & 1 & & &
\end{pmatrix},\  \begin{pmatrix} 1 & 1 & 2 & 2 & 1 & 1\\ & & 1 & & & \end{pmatrix},\]
\[ \begin{pmatrix} 0 & 1 & 2 & 2 & 2 & 1\\ & & 1 & & &
\end{pmatrix},\  \begin{pmatrix} 0 & 1 & 1 & 1 & 1 & 1 \\ & & 1 & & & \end{pmatrix}.\]
Then we have $\tau(\omega_6 )+\omega_6={\underset{1\le i\le
4}{\sum}}\beta_i$, and $m_6=4$.

\vs.2cm\ni \textbf{Case 7:} $d=7$.

We have  $\tau(\omega_7 )+\omega_7 = \begin{pmatrix} 2 & 4 & 6 & 5
& 4 & 3\\ & & 3 & & &
\end{pmatrix}$.  We let $\{\beta_1,\beta_2,\beta_3\}$ be
the unordered triple of roots:
\[\begin{pmatrix} 0 & 1 & 2 & 2 & 2 & 1 \\ & & 1 & & & \end{pmatrix},
\ \begin{pmatrix} 1 & 2 & 2 & 2 & 1 & 1\\ & & 1 & & &
\end{pmatrix},\ \begin{pmatrix} 1 & 1 & 2 & 1 & 1 & 1 \\
& & 1 & & & \end{pmatrix}.\] Then we have $\tau(\omega_7
)+\omega_7=\beta_1+\beta_2+\beta_3$, and $m_7=3$.

We have ${\underset{1\le d\le 7}{\sum}}m_d=27$ which is equal to
codim$_{G/B}P/B$. Hence we obtain
\begin{thm}\label{main3}
The LMP conjecture holds for the minuscule $G/P$, $G,P$ being as
above.
 \end{thm}

%%%%%%%%%%%%%%%%%%%%%%%%%%%%%%%%%%%%%%%
\section{The remaining minuscule $G/P$'s}\label{last}
In this section, we give the details for the remaining $G/P$'s
along the same lines as in \S \ref{dn}, \S \ref{e6}, \S \ref{e7},
thus providing an alternate proof for the results of
\cite{mp,lrs}. We fix a maximal parabolic subgroup $P$ of $G$,
denote (as in the previous sections), the element of largest
length in $W_P$ by $\tau$. Also, as in the previous sections (cf.
Definition \ref{md}), we shall denote $ord_{e}\,p_{\tau
w_0^{(d)}}(=ord_{\tau}\,p_{w_0^{(d)}})$ by $m_d,1\le d\le n$.

\subsection{The simple root $\alpha$}\label{simple} While computing
$m_d$, as in \S \ref{dn}, \S \ref{e6}, \S \ref{e7}, in each case,
we work with a simple root $\alpha$ which occurs with a non-zero
coefficient $c_\alpha$ in the expression for
$\omega_d+\tau(i(\omega_d))$ (as a non-negative integral linear
combination of simple roots) and which has the property that in
the expression for any positive root (as a non-negative integral
linear combination of simple roots), it occurs with a coefficient
$\le 1$. This $\alpha$ will depend on the type of $G$, and we
shall specify it in each case. Then as seen in \S \ref{compute},
$m_d\ge c_\alpha$. We shall first exhibit a set of roots
$\beta_1,\cdots,\beta_r, r=c_\alpha$ in
$R^+\,\setminus\,R^+_{P_{d}}$, satisfying (a)-(c) in Step 2 of \S
\ref{compute}, and then conclude that $m_d=c_\alpha$. It will turn
out (as shown below) that in all cases, ${\underset{1\le d\le
n}{\sum}}m_d$ equals codim$\,_{G/B}P/B$, thus proving the LMP
conjecture (and hence Wahl's conjecture).
\subsection{Grassmannian} Let $G=SL(n)$. In this case, every
maximal parabolic subgroup is minuscule. Let us fix a maximal
parabolic subgroup $P:=P_c$; we may suppose that $c\le n-c$ (in
view of the natural isomorphism $G/P_{c}\cong G/P_{n-c}$).
Identifying the Weyl group with the symmetric group $S_n$, we have
\[\tau=(c\,c-1\cdots 1\,n\,n-1\cdots c+1).\]  Let
$\epsilon_j,1\le j\le n$ be the character of $T$ (the maximal
torus consisting of diagonal matrices in $G$), sending a diagonal
matrix to its $j$-th diagonal entry. Note that ${\underset{1\le
j\le n}{\sum}}\epsilon_j=0$ (writing the elements of the character
group additively, as is customary); this fact will be repeatedly
used in the discussion below. Also, for $1\le d\le n-1$, we have
$i(\omega_d)=\omega_{n-d}$. We observe that in the expression for
a positive root (as a non-negative integral linear combination of
simple roots), any simple root occurs with a coefficient $\le 1$.
For each $P_d$, we shall take $\alpha$ to be $\alpha_d$.

\vs.2cm\ni\textbf{Case 1:} Let $d<c$. We have (cf.\cite{bou})
\begin{eqnarray*}\omega_d+\tau(i(\omega_d)) & = & \omega_d+\tau(\omega_{n-d})\\
 & = & (\epsilon_1+\cdots +\epsilon_d) + (\epsilon_1+\cdots +\epsilon_c+
\epsilon_n+\cdots +\epsilon_{c+d+1})\\
 & = & \epsilon_1+\cdots +\epsilon_d-(\epsilon_{c+1}+\cdots +\epsilon_{c+d})\\
 & = & (\epsilon_1-\epsilon_{c+d})+
(\epsilon_2-\epsilon_{c+d-1})+\cdots+(\epsilon_d-\epsilon_{c+1})
\end{eqnarray*} (note that $d<c\le n-c$, and hence $n-d>c$). From the last expression,
it is clear that $c_\alpha=d$ ($c_\alpha$ being as in \S
\ref{simple}); note that every one of the roots in the last
expression belongs to $R^+\,\setminus\,R^+_{P_{d}}$. We now let
$\beta_1,\cdots,\beta_d$  be the unordered $d$-tuple of roots:
\[\epsilon_1-\epsilon_{c+d},
\epsilon_2-\epsilon_{c+d-1},\cdots,\epsilon_d-\epsilon_{c+1}.\]
Then it is easily checked that the above $\beta_j$'s satisfy
(a)-(c) in Step 2 of \S \ref{compute}, and we have $m_d=d$ (in
fact any such grouping will also work).

\vs.2cm\ni\textbf{Case 2:} Let $c\le d\le n-c$. We have
\begin{eqnarray*} \omega_d+\tau(i(\omega_d)) & = & (\epsilon_1+\cdots +\epsilon_d) + (\epsilon_1+\cdots +\epsilon_c+
\epsilon_n+\cdots +\epsilon_{c+d+1})\\ & = & \epsilon_1+\cdots
+\epsilon_c-(\epsilon_{d+1}+\cdots +\epsilon_{d+c})\\
 & = & (\epsilon_1-\epsilon_{d+c})+
(\epsilon_2-\epsilon_{d+c-1})+\cdots+(\epsilon_c-\epsilon_{d+1}).
\end{eqnarray*}
Hence $c_\alpha=c$. We now let
$\beta_1,\cdots,\beta_c$  be the unordered $c$-tuple of roots:
\[\epsilon_1-\epsilon_{d+c},
\epsilon_2-\epsilon_{d+c-1},\cdots,\epsilon_c-\epsilon_{d+1}.\]
Then it is easily checked that the above $\beta_j$'s satisfy
(a)-(c) in Step 2 of \S \ref{compute}, and we have $m_d=c$.

\vs.2cm\ni\textbf{Case 3:} Let $n-c<d\le n-1$. We have
\begin{eqnarray*} \omega_d+\tau(i(\omega_d)) & = & (\epsilon_1+\cdots +\epsilon_d) +
(\epsilon_c+\epsilon_{c-1}+\cdots
+\epsilon_{c+d+1-n})\\
 & = & -(\epsilon_{d+1}+\cdots +\epsilon_n)+(\epsilon_c+\cdots
+\epsilon_{c+d+1-n})\\  & = & (\epsilon_c-\epsilon_{d+1})+
(\epsilon_{c-1}-\epsilon_{d+2})+\cdots+(\epsilon_{c+d+1-n}-
\epsilon_{n})
\end{eqnarray*} (note that $c\le n-c<d$). Hence $c_\alpha=n-d$.
We now let $\beta_1,\cdots,\beta_{n-d}$ be the unordered
$(n-d)$-tuple of roots: \[\epsilon_c-\epsilon_{d+1},
\epsilon_{c-1}-\epsilon_{d+2},\cdots,\epsilon_{c+d+1-n}-
\epsilon_{n}.\] Then it is easily checked that the above choice of
$\beta_j$'s satisfy (a)-(c) in Step 2 of \S \ref{compute}, and we
have $m_d=n-d$.

From the above computations, we have,
\[{\underset{1\le d\le
n-1}{\sum}}m_d=(1+\cdots+c-1)+(n-2c+1)c+(1+\cdots+c-1)=c(n-c)=
codim_{G/B}P/B.\]
\subsection{Lagrangian Grassmannian}\label{cn} Let $V=K^{2n}$ together with a
nondegenerate, skew-symmetric bilinear form $(\cdot  , \cdot )$.
Let $H=SL(V)$ and $G=Sp(V)=\{A\in SL(V) \mid A$ leaves the form
$(\cdot  , \cdot )$ invariant $\}$. Taking the matrix of the form
(with respect to the standard basis $\{ e_1,...,e_{2n} \}$ of $V$)
to be
$$E=\begin{pmatrix}
           0  &  J  \\
           -J &  0
\end{pmatrix}$$
where $J$ is the anti-diagonal $(1,\ldots ,1)$ of size $n\times
n$, we may realize $Sp(V)$ as the fixed point set of a certain
involution $\sigma$ on $SL(V),$ namely $G=H^{\sigma}$, where
$\sigma: H \longrightarrow H$ is given by
$\sigma(A)=E(^t\!\!A)^{-1}E^{-1}$. Denoting by $T_H$ (resp. $B_H$)
the maximal torus in $H$ consisting of diagonal matrices (resp.
the Borel subgroup in $H$ consisting of upper triangular matrices)
we see easily that $T_H, B_H$ are stable under $\sigma$.  We set
$T_G={T^{\sigma}_H} , B_G={B^{\sigma}_H}$. Then it is well known
that $T_G$ is a maximal torus in $G$ and $B_G$ is a Borel subgroup
in $G$. We have a natural identification of the Weyl group $W$ of
$G$ as a subgroup of $S_{2n}:$
\[W=\{(a_1 \cdots a_{2n}) \in S_{2n} \mid a_i=2n+1-a_{2n+1-i},\ 1
\leq i \leq 2n\}.\]   Thus $w=(a_1\cdots a_{2n}) \in W_G$ is known
once $(a_1 \cdots a_n)$ is known.

For details see \cite{lr}.

 In the sequel, we shall denote such a
$w$ by just $(a_1 \cdots a_n)$; also, for $1\le i\le 2n$, we shall
denote $2n+1-i$ by $i'$. The Weyl involution is the identity map.
In type \textbf{C}$_n$, we have that in the expression for a
positive root (as a non-negative integral linear combination of
simple roots), the simple root $\alpha_n$ occurs with a
coefficient $\le 1$. For all $P_d, 1\le d\le n$, we shall take
$\alpha$ (cf. \S \ref{simple}) to be $\alpha_n$.

Let $P=P_n$. Then $\tau=(n\cdots 1)$.
 Let $1\le d\le n$. We have
\[\omega_d+\tau(\omega_d) =(\epsilon_1+\cdots +\epsilon_d) +
(\epsilon_n+\cdots +\epsilon_{n+1-d}).
\] If $d< n+1-d$, then each $\epsilon_j$
in the above sum are distinct. Hence writing
\[\omega_d+\tau(\omega_d)=(\epsilon_1+\epsilon_{n})+
(\epsilon_{2}+\epsilon_{n-1})+\cdots+(\epsilon_{d}+
\epsilon_{n+1-d})\] we have that each root in the last sum belongs
to $R^+\,\setminus\,R^+_{P_{d}}$ (since each of the roots clearly
involves $\alpha_d$); further, each of the roots involves
$\alpha_n$ with coefficient equal to $1$. Hence $c_\alpha=d$. We
let $\beta_1,\cdots,\beta_{d}$  be the unordered $d$-tuple of
roots:
\[\epsilon_1+\epsilon_{n},
\epsilon_{2}+\epsilon_{n-1},\cdots,\epsilon_{d}+
\epsilon_{n+1-d}.\] If $d\ge n+1-d$, then
\begin{eqnarray*} \omega_d+\tau(\omega_d) & = & (\epsilon_1+\cdots +\epsilon_{n-d}) +
(\epsilon_n+\cdots
+\epsilon_{d+1})+(2\epsilon_{n-d+1}+\cdots+2\epsilon_d)\\
 & = & (\epsilon_1+\epsilon_{n})+
(\epsilon_{2}+\epsilon_{n-1})+\cdots+(\epsilon_{n-d}+
\epsilon_{d+1})+ 2\epsilon_{n-d+1}+\cdots+2\epsilon_d .
\end{eqnarray*} Again, we have that each root in the last sum belongs
to $R^+\,\setminus\,R^+_{P_{d}}$, and involves $\alpha_n$ with
coefficient equal to $1$. Hence we obtain that $c_\alpha=d$. We
let $\beta_1,\cdots,\beta_{d}$  be the unordered $d$-tuple of
roots:
\[\epsilon_1+\epsilon_{n},
\epsilon_{2}+\epsilon_{n-1},\cdots,\epsilon_{n-d}+ \epsilon_{d+1},
2\epsilon_{n-d+1},\cdots,2\epsilon_d.
\] Then it is easily checked that in both cases, the  $\beta_j$'s
satisfy (a)-(c) in Step 2 of \S \ref{compute}, and we have
$m_d=d$.

 Hence \[{\underset{1\le d\le n}{\sum}}m_d={\underset{1\le
d\le n}{\sum}}d={n+1\choose 2}=codim_{G/B}P/B.\]

\subsection{Orthogonal Grassmannian}\label{ortho}
Since $SO(2n+1)/P_n\cong SO(2n+2)/P_{n+1}$, and $SO(2n)/P_n\cong
SO(2n)/P_{n-1}$, we shall give the details for the orthogonal
Grassmannian $SO(2n)/P_n$. Thus $G=SO(2n),P=P_n$, and
$\tau=(n\,n-1\cdots 1)$.

In type \textbf{D}$_n$, we have that in the expression for a
positive root (as a non-negative integral linear combination of
simple roots), the simple root $\alpha_n$ occurs with a
coefficient $\le 1$. For all $P_d, 1\le d\le n$, we shall take
$\alpha$ (cf. \S \ref{simple}) to be $\alpha_n$.

\vs.2cm\ni\textbf{Case 1:} Let $1\le d\le n-2$. We have,
\[\omega_d+\tau(\omega_d) =(\epsilon_1+\cdots +\epsilon_d) +
(\epsilon_n+\cdots +\epsilon_{n+1-d}).
\] If $d< n+1-d$, then as in \S \ref{cn}, we have
 \[\omega_d+\tau(\omega_d)=(\epsilon_1+\epsilon_{n})+
(\epsilon_{2}+\epsilon_{n-1})+\cdots+(\epsilon_{d}+
\epsilon_{n+1-d}).\] We have that each root in the last sum belongs
to $R^+\,\setminus\,R^+_{P_{d}}$, and involves $\alpha_n$ with
coefficient equal to $1$. Hence $c_\alpha=d$. We let
$\beta_1,\cdots,\beta_{d}$  be the unordered $d$-tuple of roots:
\[\epsilon_1+\epsilon_{n},
\epsilon_{2}+\epsilon_{n-1},\cdots,\epsilon_{d}+
\epsilon_{n+1-d}.\] If $d\ge n+1-d$, then
\begin{eqnarray*}\omega_d+\tau(\omega_d) & = & (\epsilon_1+\cdots +\epsilon_{n-d}) +
(\epsilon_n+\cdots
+\epsilon_{d+1})+(2\epsilon_{n-d+1}+\cdots+2\epsilon_d)\\
& = & (\epsilon_1+\epsilon_{n+1-d})+
(\epsilon_{2}+\epsilon_{n+2-d})+\cdots+(\epsilon_{d}+
\epsilon_{n}). \end{eqnarray*} Again we have that  each root in
the last sum belongs to $R^+\,\setminus\,R^+_{P_{d}}$, and
involves $\alpha_n$ with coefficient equal to $1$. Hence we obtain
that $c_\alpha=d$. We let $\beta_1,\cdots,\beta_{d}$  be the
unordered $d$-tuple of roots:
\[(\epsilon_1+\epsilon_{n+1-d}),
(\epsilon_{2}+\epsilon_{n+2-d}),\cdots,(\epsilon_{d}+
\epsilon_{n}).
\] Then it is easily checked that in both cases, the  $\beta_j$'s
satisfy (a)-(c) in Step 2 of \S \ref{compute}, and we have
$m_d=d$.

\vs.2cm\ni\textbf{Case 2:} $d=n-1$. We have,
\[\omega_{n-1}+\tau(i(\omega_{n-1}))=
\begin{cases}(\epsilon_2+\cdots+\epsilon_{n-1}), &
\mathrm{\ if\ } n\mathrm{\ is\ even}\\
(\epsilon_1+\cdots+\epsilon_{n-1}), & \mathrm{\ if\ } n\mathrm{\
is\ odd}.
\end{cases}\] Hence expressing $\omega_{n-1}+\tau(i(\omega_{n-1}))$ as a
non-negative linear integral combination of positive roots, we
obtain
\[\omega_{n-1}+\tau(i(\omega_{n-1}))=
\begin{cases}
{\underset{2\le i\le
{\frac{n}{2}}}{\sum}}\epsilon_i+\epsilon_{n+1-i}, &
\mathrm{\ if\ } n\mathrm{\ is\ even}\\
{\underset{1\le i\le
{\frac{n-1}{2}}}{\sum}}\epsilon_i+\epsilon_{n-i}, & \mathrm{\ if\
} n\mathrm{\ is\ odd}.
\end{cases}\]
 Thus $\omega_{n-1}+\tau(i(\omega_{n-1}))$ is a sum of
 ${\frac{n-2}{2}}$ or ${\frac{n-1}{2}}$ roots in
 $R^+\,\setminus\,R^+_{P_{n-1}}$, according as $n$ is even or odd;
 further, each of them involves
 $\alpha_n$ with  coefficient one. Hence
 $$c_\alpha=\begin{cases}{\frac{n-2}{2}}&
\mathrm{\ if\ } n\mathrm{\ is\ even}\\
{\frac{n-1}{2}}& \mathrm{\ if\ } n\mathrm{\ is\ odd}.
\end{cases}$$

 We let
$\beta_1,\cdots,\beta_{r}, r=c_\alpha$  be the unordered $r$-tuple
of roots:\[\beta_i=\begin{cases}\epsilon_i+\epsilon_{n+1-i}, 2\le
i\le {\frac{n}{2}},& \mathrm{\ if\ } n\mathrm{\ is\ even}\\
\epsilon_i+\epsilon_{n-i}, 1\le i\le {\frac{n-1}{2}},& \mathrm{\
if\ } n\mathrm{\ is\ odd}.
\end{cases}\]

Clearly, the above  $\beta_j$'s  satisfy (a)-(c) in Step 2 of \S
\ref{compute}
$$m_{n-1}=\begin{cases}{\frac{n-2}{2}},&
\mathrm{\ if\ } n\mathrm{\ is\ even}\\
{\frac{n-1}{2}},& \mathrm{\ if\ } n\mathrm{\ is\ odd}.
\end{cases}$$

\vs.2cm\ni\textbf{Case 3:} $d=n$. Proceeding as in Case 2, we
have,
\[\omega_{n}+\tau(i(\omega_{n}))=
\begin{cases}(\epsilon_1+\cdots+\epsilon_{n}), &
\mathrm{\ if\ } n\mathrm{\ is\ even}\\
(\epsilon_2+\cdots+\epsilon_{n}), & \mathrm{\ if\ } n\mathrm{\ is\
odd}.
\end{cases}\] Hence we obtain
\[\omega_{n}+\tau(i(\omega_{n}))=
\begin{cases}
{\underset{1\le i\le
{\frac{n}{2}}}{\sum}}\epsilon_i+\epsilon_{n+1-i}, &
\mathrm{\ if\ } n\mathrm{\ is\ even}\\
{\underset{2\le i\le
{\frac{n+1}{2}}}{\sum}}\epsilon_i+\epsilon_{n+2-i}, & \mathrm{\
if\ } n\mathrm{\ is\ odd}.
\end{cases}\]
 Thus $\omega_{n-1}+\tau(i(\omega_{n-1}))$ is a sum of
 ${\frac{n}{2}}$ or ${\frac{n-1}{2}}$ roots in
 $R^+\,\setminus\,R^+_{P_{n}}$, according as $n$ is even or odd;
 further, each of them involves
 $\alpha_n$ with  coefficient one. Hence
 $$c_\alpha=\begin{cases}{\frac{n}{2}},&
\mathrm{\ if\ } n\mathrm{\ is\ even}\\
{\frac{n-1}{2}},& \mathrm{\ if\ } n\mathrm{\ is\ odd}.
\end{cases}$$

We let
$\beta_1,\cdots,\beta_{r}, r=c_\alpha$  be the unordered $r$-tuple
of roots:\[\beta_i=\begin{cases}\epsilon_i+\epsilon_{n+1-i}, 1\le
i\le {\frac{n}{2}},& \mathrm{\ if\ } n\mathrm{\ is\ even}\\
\epsilon_i+\epsilon_{n+2-i}, 2\le i\le {\frac{n+1}{2}},& \mathrm{\
if\ } n\mathrm{\ is\ odd}.
\end{cases}\]

Clearly, the above choice of $\beta_j$'s satisfies
(a)-(c) in Step 2 of \S \ref{compute}
$$m_{n}=\begin{cases}{\frac{n}{2}},&
\mathrm{\ if\ } n\mathrm{\ is\ even}\\
{\frac{n-1}{2}},& \mathrm{\ if\ } n\mathrm{\ is\ odd}.
\end{cases}$$

Combining cases 2 and 3, we obtain $m_{n-1}+m_n=n-1$, in both even
and odd cases. Combining this with the value of $m_d, 1\le d\le
n-2$ as obtained in case 1, we obtain \[{\underset{1\le d\le
n}{\sum}}m_d={\underset{1\le d\le n-2}{\sum}}d+n-1={n\choose
2}=codim_{G/B}P/B.\]

Thus combining the results of \S \ref{dn}, \S \ref{e6}, \S
\ref{e7}, \S \ref{last}, we obtain
\begin{thm}
The LMP conjecture and Wahl's conjecture hold for a minuscule $G/P$.
\end{thm}

\begin{rem}\label{path2}
 Thus in these cases again, we obtain a nice realization for $ord_{\tau}\,p_{w_0^{(d)}}$
 (the order of vanishing along $P/B$ of $p_{w_0^{(d)}}$) as being the length of
the shortest path through extremal weights in the weight lattice
connecting the highest weight (namely, $i(\omega_d)$ in
$H^0(G/B,L(\omega_d))$ and the extremal weight $-\tau(\omega_d)$.
\end{rem}

\end{document}